\documentclass[oneside,12pt]{amsart}
\usepackage{eurosym}
\usepackage{amsfonts} 
\usepackage{amsmath,amssymb,amsthm,xcolor}
\usepackage{latexsym}
\usepackage{float}
\usepackage{placeins}
\usepackage{bm}
\usepackage[utf8]{inputenc}
\usepackage[british]{babel}
\usepackage{a4wide}
\usepackage{pdfsync}
\usepackage[colorlinks=true,linkcolor=blue,citecolor=blue]{hyperref}
\usepackage[all]{xy}
\usepackage{parskip}
\usepackage{multirow}
\usepackage{array}
\usepackage[toc,page]{appendix}
\usepackage{graphicx}
\usepackage{comment}
\usepackage[mathscr]{euscript}

\theoremstyle{definition}
\newtheorem{lemma}{Lemma}

\newtheorem{theorem}[lemma]{Theorem}
\newtheorem*{theorem*}{Theorem}
\newtheorem{corollary}[lemma]{Corollary}
\newtheorem{definition}[lemma]{Definition}

\newtheorem{remark}[lemma]{Remark}

\begin{document}

\title[On the stability of Einstein metrics carrying a special twisted spinor]{On the stability of Einstein metrics \\ carrying a special twisted spinor}

\author{Diego Artacho}
\address{D.~Artacho: Department of Mathematics, KU Leuven,
Celestijnenlaan 200B, 3001 Leuven, Belgium}
\email{diego.artachodeobeso@kuleuven.be}

\begin{abstract}
   We prove linear semi-stability for a large class of Einstein metrics of non-positive scalar curvature. More precisely, we show that any Einstein $n$-manifold with non-positive scalar curvature carrying a parallel twisted pure spin$^r$ spinor is linearly semi-stable, under mild restrictions on $n$ and $r$. We thus extend the parallel spin and spin$^c$ stability results of Dai--Wang--Wei. As an application, our result implies linear semi-stability for all negative quaternion-K{\"a}hler manifolds of dimension greater than $8$. 
\end{abstract}

\maketitle

\section{Introduction} \label{sec:intro}

A Riemannian metric $g$ on a smooth manifold $M$ is said to be Einstein if its Ricci curvature is a constant multiple of the metric,
\[
\mathrm{Ric}_g = E\, g ,
\]
for some real constant $E$ \cite{B87}. For compact manifolds, Einstein metrics are precisely the critical points of the Einstein--Hilbert functional
\[
\mathcal{S} : \mathcal{M}_1 \to \mathbb{R},
\qquad 
\mathcal{S}(g) := \int_M \mathrm{scal}_g \lvert\mathrm{vol}_g\rvert 
\]
on the space $\mathcal{M}_1$ of unit-volume Riemannian metrics, where $\lvert \mathrm{vol}_g \rvert$ denotes the volume density determined by $g$.

Given such a metric $g$, a fundamental problem in geometric analysis is to understand the second variation of $\mathcal{S}$ at $g$. We refer to the recent survey \cite{SS25} for a comprehensive overview. An Einstein metric is called \emph{stable} if it is a local maximum of $\mathcal{S}$ restricted to the space
\[
\mathcal{C} := \{ g \in \mathcal{M}_1 \mid \mathrm{scal}_g \ \text{is constant on } M \}.
\]
The linearisation of this notion leads to the definition of \emph{linear (semi-)stability}: an Einstein metric $g$ is \emph{linearly stable} (resp.\ \emph{semi-stable}) if the second variation $\mathcal{S}''_g$ is negative-definite (resp.\ negative semi-definite) on the space of transverse-traceless tensors
\[
\mathscr{S}^2_{\mathrm{tt}}
  := \{ h \in \mathscr{S}^2(M) \mid \mathrm{tr}_g h = 0,\ \delta_g h = 0 \},
\]
where $\mathscr{S}^2(M):=\Gamma(\mathrm{Sym}^2 T^*M)$ and $\delta_g$ denotes the divergence.

It is well known that linear semi-stability is governed by the spectrum of the so-called \emph{linearised Einstein operator}; namely, it corresponds to 
\[
\nabla^*\nabla - 2\mathring{R} \ge 0 
\]
on $\mathscr{S}^2_{\mathrm{tt}}$, where $\nabla$ is the Levi-Civita connection of $g$ and $\mathring{R}$ denotes the curvature action on symmetric $2$-tensors -- see~\eqref{eq:ring}. Linear stability corresponds to the inequality being strict. 

Several conditions for stability have been obtained in terms of the curvature \cite{K78,F79,BCG91,IN05,K15,BK25}. A long-standing conjecture in the field asserts that every Einstein metric of non-positive scalar curvature is stable. This problem has attracted significant attention, with substantial progress in several geometric settings.  
For example:
\begin{itemize}
\item in the presence of a parallel spinor, Dai--Wang--Wei \cite{DWW05} proved linear semi-stability by exploiting a spinorial Bochner identity; in fact, full $\mathcal{S}$-stability holds in this setting -- see \cite[Cor.~5.7]{SS25}; 
\item in the K{\"a}hler--Einstein case, an extension of this method using spin$^{c}$ geometry yields linear semi-stability in the negative scalar curvature regime \cite{DWW07};
\item more recently, Kr{\"o}ncke--Semmelmann \cite{KS24} proved strict stability for negative quaternion-K{\"a}hler manifolds using representation-theoretic methods. 
\end{itemize}

In this paper, we prove linear semi-stability for a broad class of Einstein metrics of non-positive scalar curvature, namely those admitting a \emph{parallel twisted pure spin$^r$ spinor}. Our main result is as follows.

\begin{theorem}\label{thm:main}
Let $(M^n,g)$ be a Riemannian manifold carrying a parallel twisted pure spin$^r$ spinor, for some $r \ge 3$, $r\neq 4$, $n\neq 8$, and $n+4r-16 \neq 0$. Then $g$ is Einstein and, if $\mathrm{scal}_g \le 0$, it is linearly semi-stable. \qed
\end{theorem}

The class of manifolds that carry parallel twisted pure spin$^r$ spinors will be introduced in Section \ref{sec:prelim}. For now, let us remark that this is a vast family: it includes, for example, all even-dimensional manifolds with special Riemannian holonomy \cite[Sec. 4]{HS19}. 

Our proof generalises the strategy initiated by Dai--Wang--Wei: we construct an isometric bundle map from symmetric $2$-tensors to twisted spinor-valued $1$-forms, and use it to relate the linearised Einstein operator to the square of a Dirac-type operator. We do this in Section \ref{sec:parallel}. The key new ingredient comes from the algebraic identities satisfied by the $2$-forms associated to twisted pure spinors \cite{HS19}, which yield a curvature relation that allows us to control the extra terms appearing in a general Bochner-like formula that we obtain in Section \ref{sec:parallel}. This makes it possible to derive the required lower bound for the linearised Einstein operator in the non-positive scalar curvature case, which we accomplish in Section \ref{sec:pure}. 

\section{Preliminaries}
\label{sec:prelim}

Let $(M,g)$ be an $n$-dimensional Riemannian manifold. Throughout the paper, we will work locally: we fix a point $x \in M$ and a local orthonormal frame $(e_1 , \dots , e_n)$ around it such that $\nabla e_i (x) = 0$ for all $i$. 

For vector fields $X,Y$, the Riemannian curvature endomorphism is given by  
\[
R_{XY} := \left[ \nabla_X , \nabla_Y \right] - \nabla_{[X,Y]} ,  
\]
and for $i,j,k,l \in \{1 , \dots n \}$
\[
R_{ijkl} := g \left( R_{e_i e_j} e_k , e_l \right) . 
\]
The Ricci and scalar curvature are given, respectively, by 
\[
\mathrm{Ric}_{ij} := \mathrm{Ric}_{g}(e_i,e_j) = R_{ikkj} , \qquad \mathrm{scal}_g = R_{ikki} , 
\]
where we adopt the convention that a summation is taken whenever indices are repeated. 

The curvature operator acts on symmetric $2$-tensors by 
\begin{equation}\label{eq:ring}
(\mathring{R} h) (X,Y) := h \left( R_{X e_i} e_i , Y \right)   . 
\end{equation}

We say that $M$ is \emph{spin$^r$} if there exists an oriented Euclidean rank-$r$ vector bundle $F \to M$ such that $TM \oplus F$ is spin \cite{EH16,AM21}. Such an $F$ is called an \emph{auxiliary vector bundle} for $TM$. If $\Sigma M$ and $\Sigma F$ denote the locally defined complex spinor bundles of $TM$ and $F$ respectively, the fact that $TM \oplus F$ is spin implies that 
\[
\Sigma(M,F,m) := \Sigma M \otimes (\Sigma F)^{\otimes m}
\]
is an honest vector bundle over $M$ for all odd values of $m$. If $M$ is itself spin, even values of $m$ are also allowed. This bundle is equipped with a natural Hermitian product, which we denote by $\langle \cdot , \cdot \rangle$. Sections of these bundles are called \emph{spin$^r$ spinors}, \emph{twisted spinors} or simply \emph{spinors}, and they can be Clifford-multiplied by sections of $TM \oplus F$ in a natural way. 

Each metric connection $\theta$ on $F$, together with the Levi-Civita connection on $TM$, uniquely determines a Hermitian connection $\nabla^{\theta}$ on $\Sigma(M,F,m)$. 

We now fix a local orthonormal frame $(f_1 , \dots , f_r)$ of $F$ around $x$. Each spinor $\psi \in \Gamma(\Sigma(M,F,m))$ determines a $2$-form on $M$ for each $1 \le k,l \le r$: 
\begin{equation*}
\eta^{\psi}_{kl}(X,Y) := \Re \left\langle X \wedge Y f_k f_l \cdot \psi , \psi \right\rangle , 
\end{equation*}
where $X \wedge Y := XY + g(X,Y)$ and $\Re$ denotes the real part.  

Every section of $\omega$ of $T^* M \otimes T^*M$ naturally defines a section of $\mathrm{End}(TM) = T^*M \otimes TM$ by $\hat{\omega}(X) := \omega(X,\cdot)^{\sharp}$, which locally reads
\[
g(\hat{\omega}(e_i) , e_j) = \omega( e_i , e_j ) . 
\]

We now define the central property of spinors that we study in this paper: 
\begin{definition}[{\cite[Def. 3.5]{HS19}}] \label{def:pure}
    A spinor $\psi \in \Gamma(\Sigma(M,F,m))$ is said to be \emph{twisted pure} if the following two conditions are satisfied for all $1 \le k < l \le r$: 
    \begin{itemize}
        \item $\hat{\eta}^{\psi}_{kl} \circ \hat{\eta}^{\psi}_{kl} = -\mathrm{Id}$, and 
        \item $r = 2$ or $\left(\eta^{\psi}_{kl} + 2f_k f_l \right) \cdot \psi = 0$.  
    \end{itemize}
\end{definition}

Twisted pure spinors are of interest because they can be used to characterise Riemannian manifolds with special holonomy, as shown in \cite[Sec. 4]{HS19}. They generalise the classical notion of pure spinor for $r=2$ used to describe K{\"a}hler holonomy.   

\section{A Bochner-type formula}
\label{sec:parallel}

Throughout this section, $(M^n, g)$ is a compact spin$^r$ Einstein manifold with auxiliary vector bundle $F$ that admits a non-zero spin$^r$ spinor $\psi \in \Gamma (\Sigma (M,F,m))$ that is parallel with respect to an auxiliary connection $\theta$ with curvature $2$-form $\Theta$. For simplicity, we denote $\nabla = \nabla^{\theta}$. Since $\psi$ necessarily has constant length, we may assume without loss of generality that $\lvert \psi \rvert = 1$. 

Following the classical ideas of \cite{DWW05,DWW07}, we define a bundle map from the space of symmetric $2$-tensors to the space of $\Sigma (M,F,m)$-valued $1$-forms: 
\begin{equation}\label{eq:def_phi}
\begin{aligned}
\Phi \colon \mathrm{Sym}^2(T^*M) &\to \Sigma (M,F,m) \otimes T^*M \\
h &\mapsto \left( X \mapsto \hat{h}(X) \cdot \psi \right) .
\end{aligned}
\end{equation}
With respect to a local orthonormal frame $(e_1 , \dots , e_n)$, it takes the form
\[
\Phi(h) = h_{ij} e_i \cdot \psi \otimes e^j , 
\]
This map is easily seen to be  parallel and isometric. 

Let $\mathcal{D}$ be the Dirac operator on $\Sigma M \otimes T^*M$, given locally by 
\[
\mathcal{D} = e_i \cdot \nabla_{e_i} . 
\]
We now prove a relationship between the linearised Einstein operator $\nabla^* \nabla - 2 \mathring{R}$ on symmetric $2$-tensors and $\mathcal{D}^* \mathcal{D}$ via the bundle map $\Phi$.  

\begin{lemma}\label{lemma:dirac}
For each symmetric $2$-tensor $h$ on $M$, 
\begin{equation*}
    \mathcal{D}^* \mathcal{D} \Phi(h) = \Phi \left( \nabla^* \nabla h - 2 \mathring{R} h + (\mathrm{Ric} \circ h) \right) - \frac{1}{2} h_{ip} \left(\hat{\Theta}_{kl}\right)_{pj} e_i f_k f_l \cdot \psi \otimes e^j ,  
\end{equation*}
where $\Phi$ is interpreted as in \eqref{eq:def_phi}, extended in the obvious way to general (not necessarily symmetric) $2$-tensors.
\end{lemma}
\begin{proof}
Using that $\mathcal{D}^* = \mathcal{D}$ and that $\Phi$ is parallel, we compute 
\begin{align*}
    \mathcal{D}^* \mathcal{D} (\Phi(h)) &= \mathcal{D} \left( e_l \cdot \nabla_{e_l} \Phi(h) \right) = e_k \cdot \nabla_{e_k} \left( e_l \cdot \Phi (\nabla_{e_l} h) \right) \\
    &= e_k \cdot \left( (\nabla_{e_k} e_l) \cdot \Phi(\nabla_{e_l} h) + e_l \cdot \nabla_{e_k} \Phi(\nabla_{e_l} h)  \right) \\
    &= e_k e_l \cdot \nabla_{e_k} \Phi(\nabla_{e_l} h) = e_k e_l \cdot \Phi (\nabla_{e_k} \nabla_{e_l} h) \\
    &= \left( \nabla_{e_k} \nabla_{e_l} h \right) (e_i , e_j) e_k e_l e_i \cdot \psi \otimes e^j \\
    &= \left( - \left( \nabla_{e_k} \nabla_{e_k} h \right) (e_i , e_j) e_i + \sum_{k < l} \left( \nabla_{e_k} \nabla_{e_l} h - \nabla_{e_l} \nabla_{e_k} h \right) (e_i , e_j) e_k e_l e_i \right) \cdot \psi \otimes e^j \\
    &= \left( \left( \nabla^* \nabla h \right) (e_i , e_j) e_i + \sum_{k<l} \left( h \left( R_{e_l e_k} e_i , e_j \right) + h \left( e_i , R_{e_l e_k} e_j \right) \right) e_k e_l e_i \right) \cdot \psi \otimes e^j \\
    &= \Phi( \nabla^* \nabla h) + \frac{1}{2} h(R_{e_l e_k} e_i , e_j) e_k e_l e_i \cdot \psi \otimes e^j + \frac{1}{2} h(e_i , R_{e_l e_k} e_j) e_k e_l e_i \cdot \psi \otimes e^j \\
    &= \Phi( \nabla^* \nabla h) + \frac{1}{2} h_{jp} R_{lkip} e_k e_l e_i \cdot \psi \otimes e^j + \frac{1}{2} h_{ip} R_{lkjp} e_k e_l e_i \cdot \psi \otimes e^j . 
\end{align*}
Now, using the first Bianchi identity and the Clifford relations, we can simplify the last two terms of the previous sum as follows: 
\begin{align*}
    R_{lkip} e_k e_l e_i &= R_{pikl} e_k e_l e_i = -R_{pkli} e_k e_l e_i - R_{plik} e_k e_l e_i \\
    &= -R_{pkli} e_k e_l e_i + R_{plik} e_l e_k e_i + 2 \delta_{kl} R_{plik} e_i \\
    &= -R_{pkli} e_k e_l e_i - R_{plik} e_l e_i e_k - 2\delta_{ik} R_{plik} e_l + 2R_{pkik} e_i \\ 
    &= -2R_{pkli} e_k e_l e_i - 2 \mathrm{Ric}_{pi} e_i \\
    &= -\frac{2}{3} \sum_{\substack{k,\,l,\,i \\ \text{distinct}}} \left( R_{pkli} + R_{plik} + R_{pikl} \right) e_k e_l e_i \\
    &- 2R_{pkki} e_k e_k e_i - 2R_{pklk} e_k e_l e_k - 2 \mathrm{Ric}_{pi} e_i \\
    &= -4 R_{pklk} e_l - 2\mathrm{Ric}_{pi} e_i = 2 \mathrm{Ric}_{pl} e_l , 
\end{align*}
and
\begin{align*}
    R_{lkjp} e_k e_l e_i &= - R_{jpkl} e_k e_l e_i = R_{jpkl} e_k e_i e_l + 2 \delta_{il} R_{jpkl} e_k \\
    &= -R_{jpkl} e_i e_k e_l - 2 \delta_{ik} R_{jpkl} + 2 \delta_{il} R_{jpkl} e_k \\
    &= -R_{jpkl} e_i e_k e_l - 2 R_{jpil} e_l + 2 R_{jpki} e_k \\
    &= -R_{jpkl} e_i e_k e_l - 4 R_{jpik} e_k . 
\end{align*}
Hence, we obtain 
\begin{align*}
    \mathcal{D}^* \mathcal{D} \Phi(h) &= \Phi(\nabla^* \nabla h) + h_{jp} \mathrm{Ric}_{pl} e_l \cdot \psi \otimes e^j + \frac{1}{2} h_{ip} \left( -R_{jpkl} e_i e_k e_l - 4 R_{jpik} e_k \right) \cdot \psi \otimes e^j \\
    &= \Phi(\nabla^* \nabla h) +  \left((\mathrm{Ric} \circ h)_{lj} e_l - 2 (\mathring{R}h)_{kj} e_k - \frac{1}{2} h_{ip} R_{jpkl} e_i e_k e_l \right) \cdot \psi \otimes e_j \\
    &= \Phi \left( \nabla^* \nabla h + (\mathrm{Ric} \circ h) - 2(\mathring{R}h) \right) + \frac{1}{2} h_{ip} R_{pjkl} e_i e_k e_l \cdot \psi \otimes e^j . 
\end{align*}
Finally, since $\psi$ is parallel, 
\begin{align*}
    0 = R_{pjkl} e_k e_l \cdot \psi + \left(\Theta_{kl}\right)_{pj} f_k f_l \cdot \psi , 
\end{align*}
and this finishes the proof. 
\end{proof}

\begin{corollary}\label{cor:inequality}
    For each symmetric $2$-tensor $h$ on $M$, 
    \begin{equation}\label{eq:ineq}
    \left\langle \nabla^* \nabla h - 2 \mathring{R} h , h \right\rangle \ge - \left \langle \mathrm{Ric} \circ h , h \right \rangle - \frac{1}{2} \left\langle \hat{\eta}^{\psi}_{kl} \circ h \circ \hat{\Theta}_{kl} , h \right\rangle . 
    \end{equation}
\end{corollary}
\begin{proof}
    Using Lemma \ref{lemma:dirac} and the fact that $\Phi$ is isometric, we compute
    \begin{align*}
        \left\langle \nabla^* \nabla h - 2 \mathring{R} h , h \right\rangle &= \left\langle \Phi (\nabla^* \nabla h - 2 \mathring{R} h) , \Phi(h) \right\rangle \\
        &= \left\langle \mathcal{D}^* \mathcal{D} \Phi(h) , \Phi(h) \right\rangle - \left\langle \Phi(\mathrm{Ric} \circ h) , \Phi(h) \right\rangle \\
        &+ \frac{1}{2} h_{ip} \left( \hat{\Theta}_{kl} \right)_{pj} \left\langle e_i f_k f_l \cdot \psi \otimes e^j , \Phi(h) \right\rangle \\
        &\ge - \left\langle \mathrm{Ric} \circ h , h \right\rangle + \frac{1}{2} h_{ip} \left( \hat{\Theta}_{kl} \right)_{pj} \left\langle e_i f_k f_l \cdot \psi \otimes e^j , h_{st} e_s \cdot \psi \otimes e^s \right\rangle \\
        &= - \left\langle \mathrm{Ric} \circ h , h \right\rangle + \frac{1}{2} h_{ip} h_{qj} \left( \hat{\Theta}_{kl} \right)_{pj} \left\langle e_i f_k f_l \cdot \psi \otimes e^j , e_q \cdot \psi \otimes e^j \right\rangle \\
        &= - \left\langle \mathrm{Ric} \circ h , h \right\rangle - \frac{1}{2} h_{ip} h_{qj} \left( \hat{\Theta}_{kl} \right)_{pj} \left\langle e_q e_i f_k f_l \cdot \psi , \psi \right\rangle \\
        &= - \left\langle \mathrm{Ric} \circ h , h \right\rangle - \frac{1}{2} h_{ip} h_{qj} \left( \hat{\Theta}_{kl} \right)_{pj} \left(\eta^{\psi}_{kl} \right)_{qi} \\
        &= - \left\langle \mathrm{Ric} \circ h , h \right\rangle - \frac{1}{2} \left\langle \hat{\eta}^{\psi}_{kl} \circ h \circ \hat{\Theta}_{kl} , h \right\rangle . 
    \end{align*}
\end{proof}

\begin{remark}
    Every oriented Riemannian $n$-manifold admits a natural spin$^n$ structure, given by taking the tangent bundle itself as the auxiliary vector bundle, which we endow with the Levi-Civita connection. This structure carries a non-zero parallel spin$^n$ spinor \cite[Prop. 3.3]{EH16}. An elementary calculation applying Corollary \ref{cor:inequality} to this spinor yields
    \begin{equation*}
        \left\langle \nabla^* \nabla h - 2 \mathring{R} h , h \right\rangle \ge - \left\langle \mathrm{Ric} \circ h , h \right\rangle - \left\langle \mathring{R}h , h \right\rangle . 
    \end{equation*}
    In particular, if an Einstein metric with Einstein constant $E$ satisfies $\mathring{R} < E$, then it is stable, which recovers a well-known result by Koiso \cite[Thm. 3.3]{K78}. 
\end{remark}

In the next section, we assume an additional condition on $\psi$ that will establish a relation between $\hat{\eta}^{\psi}_{kl}$ and $\hat{\Theta}_{kl}$ and will allow us to deduce stability in some cases from \eqref{eq:ineq}. 

\section{Stability of metrics with parallel twisted pure spinors}
\label{sec:pure}

We continue to assume that $(M^n, g)$ is a compact Einstein manifold endowed with a parallel spin$^r$ spinor $\psi$ with $\lvert \psi \rvert = 1$. As mentioned earlier, we will now impose a further condition on $\psi$, namely that 
\[
\psi = \frac{1}{\lvert \psi' \rvert} \psi' , 
\]
with $\psi'$ twisted pure -- see Definition \ref{def:pure}, and note that it is \emph{not} invariant by rescaling. 

Noting that there is a sign error in \cite[eq. (4.1)]{HS19}, we have the following: 

\begin{lemma}[{\cite[Lem. 4.8]{HS19}}]\label{lemma:herrera}
    Let $(M,g)$ be a Riemannian manifold carrying a parallel twisted pure spin$^r$ spinor $\psi'$, and let $\Theta = \Theta_{kl} f_k \wedge f_l$ be the curvature form of the auxiliary connection. Then, if $r \ge 3$, $r \neq 4$, $n \neq 8$, and $n+4r-16 \neq 0$, 
    \begin{equation*}
        \hat{\Theta}_{kl} = \frac{- \mathrm{scal}_g}{n(\frac{n}{4} + 2r - 4)} \hat{\eta}^{\psi'}_{kl} , 
    \end{equation*}
    for $1 \le k < l \le r$, and $g$ is Einstein. \qed
\end{lemma}

We can now prove the main result of this paper, Theorem \ref{thm:main}:  

\begin{proof}[Proof of Theorem \ref{thm:main}]
The fact that $g$ is Einstein follows from Lemma \ref{lemma:herrera}. Let $E$ be the Einstein constant, $\psi'$ a parallel twisted pure spin$^r$ spinor, $\psi = \psi' / \lvert \psi' \rvert$, and $\Theta = \Theta_{kl} f_k \wedge f_l$ the curvature $2$-form of the auxiliary connection. By Corollary \ref{cor:inequality}, 
\begin{align*}
    \left\langle \nabla^* \nabla h - 2 \mathring{R} h , h \right\rangle &\ge - \left \langle \mathrm{Ric} \circ h , h \right \rangle - \frac{1}{2} \left\langle \hat{\eta}^{\psi}_{kl} \circ h \circ \hat{\Theta}_{kl} , h \right\rangle \\
    &= -E \lvert h \rvert^2 + \frac{E}{\frac{n}{2} + 4r - 8} \left\langle \hat{\eta}^{\psi}_{kl} \circ h \circ \hat{\eta}^{\psi}_{kl} , h \right\rangle . 
\end{align*}
If $E = 0$, the result follows immediately, so suppose that $E < 0$. Then, by the triangle inequality, the Cauchy-Schwarz inequality, and the fact that $\hat{\eta}^{\psi'}_{kl} \circ \hat{\eta}^{\psi'}_{kl} = -\mathrm{Id}$, 
\begin{align*}
\left\lvert \left\langle \hat{\eta}^{\psi}_{kl} \circ h \circ \hat{\eta}^{\psi}_{kl} , h \right\rangle \right\rvert &\le \lvert h \rvert \lvert \hat{\eta}^{\psi}_{kl} \circ h \circ \hat{\eta}^{\psi}_{kl} \rvert = \frac{\lvert h \rvert}{\lvert \psi' \rvert^2} \lvert \hat{\eta}^{\psi'}_{kl} \circ h \circ \hat{\eta}^{\psi'}_{kl} \rvert \\
&= \frac{\lvert h \rvert}{\lvert \psi' \rvert^2} \mathrm{tr} \left( \left( \hat{\eta}^{\psi'}_{kl} \circ h \circ \hat{\eta}^{\psi'}_{kl} \right) \circ \left( \hat{\eta}^{\psi'}_{kl} \circ h \circ \hat{\eta}^{\psi'}_{kl} \right)^t \right)^{1/2} \\
&= \frac{\lvert h \rvert}{\lvert \psi' \rvert^2} \mathrm{tr} \left( \hat{\eta}^{\psi'}_{kl} \circ h \circ \hat{\eta}^{\psi'}_{kl} \circ \hat{\eta}^{\psi'}_{kl} \circ h \circ \hat{\eta}^{\psi'}_{kl} \right)^{1/2} \\
&= r(r-1) \frac{\lvert h \rvert}{\lvert \psi' \rvert^2} \mathrm{tr} \left( h \circ h \right)^{1/2} = r(r-1) \frac{\lvert h \rvert^2}{\lvert \psi' \rvert^2} . 
\end{align*}
And, by \cite[Thm. 3.1]{EH16}, 
\begin{equation*}
\lvert \psi' \rvert^2 \mathrm{Ric} = \hat{\eta}^{\psi'}_{kl} \circ \hat{\Theta}_{kl} = - \frac{E}{\frac{n}{4} + 2r - 4} \hat{\eta}^{\psi'}_{kl} \circ \hat{\eta}^{\psi'}_{kl} = E \, \frac{ r(r-1)}{\frac{n}{4} + 2r - 4} \mathrm{Id} ,  
\end{equation*}
which implies that 
\begin{equation*}
    \lvert \psi' \rvert^2 = \frac{r(r-1)}{\frac{n}{4} + 2r - 4} . 
\end{equation*}
Putting everything together, we see that
\begin{align*}
    \left\lvert \frac{E}{\frac{n}{2} + 4r - 8} \left\langle \hat{\eta}^{\psi}_{kl} \circ h \circ \hat{\eta}^{\psi}_{kl} , h \right\rangle \right\rvert &\le \frac{\lvert E \rvert}{\frac{n}{2} + 4r - 8} r(r-1) \frac{\lvert h \rvert^2}{\lvert \psi' \rvert^2} \\
    &= \frac{\lvert E \rvert}{\frac{n}{2} + 4r - 8} r(r-1) \lvert h \rvert^2 \frac{(\frac{n}{4} + 2r - 4)}{r(r-1)} = \lvert E \rvert \lvert h \rvert^2 , 
\end{align*}
which concludes the proof. 
\end{proof}

As an application of this result, recall that an oriented Riemannian $4m$-dimensional manifold is quaternion-K{\"a}hler if and only if it admits a parallel twisted pure spin$^3$ spinor. Hence, from Theorem \ref{thm:main} one deduces that every quaternion-K{\"a}hler manifold of negative scalar curvature of dimension greater than $8$ is linearly semi-stable. It has recently been proved by different means that, in fact, quaternion-K{\"a}hler manifolds of negative scalar curvature are strictly stable \cite{KS24}.  


\bibliographystyle{alphaurl}
\bibliography{references.bib}

\end{document}